\documentclass[11pt,a4paper]{article}
\usepackage[utf8]{inputenc}
\usepackage{amsmath}
\usepackage{amscd}
\usepackage{amsthm}
\usepackage{amssymb}
\usepackage{tikz-cd}
\usepackage{amsfonts}
\usepackage{hyperref}
\usepackage{url}
\usepackage{times}
\newtheorem{theorem}{Theorem}[section]
\newtheorem{lemma}{Lemma}[section]

\newtheorem{remark}{Remark}[section]
\allowdisplaybreaks

\title{On the distribution of the Fourier coefficients over two sparse sequences} 
\author{K. Venkatasubbareddy\\Email: \href{venkatasubbareddy7313@gmail.com}{venkatasubbareddy7313@gmail.com}
\\Department of Mathematical Sciences, IISER Berhampur,\\
Odisha, India-760003}
\date{}

\begin{document}
\maketitle

\begin{abstract}
   Let $j\geq 3$ be any fixed integer and $f$ be a primitive holomorphic cusp form of even integral weight $\kappa\geq 2$ for the full modular group $SL(2,\mathbb{Z})$. We write $\lambda_{{\rm{sym}^j  }f}(n)$ for the $n^\text{th}$ normalized Fourier coefficient of $L(s,{\rm{sym}}^j f)$. In this article, we establish asymptotic formulae for the discrete sums of the Fourier coefficients $\lambda_{{\rm{sym}}^j  f}^2(n)$ over two sparse sequence of integers, which can be written as the sum of four integral squares and the sum of six integral squares, with refined error terms.
\end{abstract}

\footnote{2020 AMS \emph{Mathematics subject classification.} Primary 11F11, 11F30, 11M06.}

\section{Introduction}
 For an even integer $\kappa\geq 2$, let $f$ be a primitive holomorphic cusp form of weight $\kappa$ for the full modular group $SL(2,\mathbb{Z})$. Throughout the paper, we refer to $f$ as a primitive holomorphic cusp form and $H_\kappa$ as the set of all primitive holomorphic cusp forms of weight $\kappa$ for the full modular group $SL(2,\mathbb{Z})$. It is well known that $f(z)$ has a Fourier series expansion at the cusp $\infty$ as
\begin{equation*}
    f(z)=\sum_{n=1}^\infty \lambda_f(n)n^{(\kappa-1)/2}e^{2\pi i n z}
\end{equation*}
for $\Im (z)>0$, where $\lambda_f(n)$ are the normalized Fourier coefficients satisfying the multiplicative property that 
\begin{equation*}
    \lambda_f(m)\lambda_f(n)=\sum_{d|(m,n)}\lambda_f(\frac{mn}{d^2}) 
\end{equation*}
for all integers $m,n\geq 1$. In 1974, Deligne \cite{Deligne} proved the Ramanujan-Petersson conjecture that $|\lambda_f(n)|\leq d(n)$, where $d(n)$ is the divisor function and which is equivalent to say that for each prime $p$, there exist two complex numbers, namely $\alpha_f(p)$ and $\beta_f(p)$ such that
\begin{equation*}
    \alpha_f(p)\beta_f(p)=|\alpha_f(p)|=|\beta_f(p)|=1 \text{ and } \lambda_f(p)=\alpha_f(p)+\beta_f(p).
\end{equation*}

The Hecke $L$-function attached to $f$ is defined as 
\begin{equation*}
    L(s, f)=\sum_{n= 1}^\infty\frac{\lambda_f(n)}{n^s}=\prod_p\left(1-\frac{\alpha_f(p)}{p^s}\right)^{-1}\left(1-\frac{\beta_f(p)}{p^s}\right)^{-1}
\end{equation*}
which converges absolutely for $\Re(s)>1.$

The $j^{\textit{th}}$ symmetric power $L$-function attached to $f$ is defined as
\begin{equation*}
        L(s, {\rm{sym}}^jf)=\prod_p\prod_{m=0}^j(1-\alpha_f(p)^{j-2m}p^{-s})^{-1}
\end{equation*}
for $\Re (s)>1$. We may express it as a Dirichlet series: for $\Re (s)>1$, 
\begin{align*}
     L(s,{\rm{sym}}^j f)&=\sum_{n=1}^\infty \frac{\lambda_{{\rm{sym}}^j f}(n)}{n^s}\nonumber\\
     &=\prod_p\bigg(1+\frac{\lambda_{{\rm{sym}}^j f}(p)}{p^s}+\ldots+\frac{\lambda_{{\rm{sym}}^j f}(p^k)}{p^{ks}}+\ldots\bigg).
\end{align*}
It is well known that $\lambda_{{\rm{sym}}^j f}(n)$ is a real multiplicative function and $\lambda_{{\rm{sym}}^j f}(p)=\lambda_f(p^j)$ for each prime $p$ and integers $j\geq 1$.

Note that $\displaystyle L(s,{\rm{sym}}^0 f)=\zeta(s)$ (Riemann zeta function) and $L(s,{\rm{sym}}^1 f)=L(s,f)$ (Hecke $L$-function).

The twisted $j^{\text{th}}$ symmetric power $L$-function attached to $f$ twisted by the Dirichlet character $\chi$ is defined as 
\begin{align*}
	L(s,\ {\rm{sym}}^{j}f\otimes\chi)=&\sum_{n=1}^\infty\frac{\lambda_{{\rm{sym}}^{j}f}(n)\chi(n)}{n^s}\\
		=&\prod_p\prod_{m=0}^{j}\left(1-\frac{\alpha_f(p)^{{j}-2m}\chi(p)}{p^s}\right)^{-1}
\end{align*}
	for $\Re(s)>1$ and $L(s,\ {\rm{sym}}^{j}f\otimes\chi)$ is of degree $j+1$.
    
For any Dirichlet character modulo $q$, the Dirichlet $L$-function is defined as 
\begin{equation*}
    L(s,\chi)=\sum_{n=1}^\infty \frac{\chi(n)}{n^s}=\prod_p \left(1-\frac{\chi(p)}{p^s}\right)^{-1}
\end{equation*}
for $\Re(s)>1$.

An important problem in number theory involves the study of the number of lattice points in a $k$-dimensional hypersphere. For $n\geq 0$, let 
\begin{equation*}
    r_k(n)=\#\{(n_1,n_2,\cdots, n_k)\in \mathbb{Z}^k: n_1^2+n_2^2+\cdots+n_k^2=n\}.
\end{equation*}
Then the formula for the sum 
\begin{equation*}
    \sum_{0\leq n\leq x}r_k(n)
\end{equation*} 
defines the lattice point number of a compact ball with origin centered and radius $\sqrt{x}$ in the $k$-dimensional space. For spheres of dimension $k\geq 4$, the situation is much easier and better understood (see \cite{Kratzel}).

The study of the average behavior of Fourier coefficients has been another interesting and important problem in number theory for a long time. Several authors have studied the average behavior of the Fourier coefficients of the above-defined $L$-functions. For example, in 1927, Hecke \cite{HeckE1827} proved that 
\begin{equation*}
    \sum_{n\leq x}\lambda_f(n)\ll x^{\frac{1}{2}}.
\end{equation*}
 After that, many researchers improved the upper estimate, such as Walfisz \cite{Walfisz} proved the upper estimate $\displaystyle{\ll x^{\frac{1+\theta}{3}}}$, Hafner and Ivi{\'c} \cite{H+I} proved $\ll x^{\frac{1}{3}}$, Rankin proved $\ll x^{\frac{1}{3}}(\log x)^{-0.0652}$, and finally the best known estimate is $\ll x^{\frac{1}{3}}(\log x)^{-0.1185}$ which is due to Wu \cite{Wu}.

In 1930, Rankin \cite{Rankin} and Selberg \cite{Selberg} independently proved that 
\begin{equation*}
    \sum_{n\leq x}\lambda_f^2(n)=c_jx+O( x^{\frac{3}{5}}). 
\end{equation*}
Recently, the exponent $\frac{3}{5}$ above has been improved to $\frac{3}{5}-\delta$ by Huang \cite{Huang} for $\delta\leq \frac{1}{560}$. This remains the best-known result in this direction. Later, many researchers have considered the higher power moments; see \cite{Fomenko1999, LauLu, LauLuWu, Lu2009(1), Lu2009(2), Lu2011}. In 2013, Zhai \cite{Zhai} considered the power sum
\begin{equation*}
   S_l(f,x) =\sum_{\substack{n=a^2+b^2\leq x\\(a,b)\in \mathbb{Z}^2}}\lambda_f(n)^l
\end{equation*}
for $2\leq l\leq 8$ and proved that $ S_l(f,x) =x \tilde{P}_l(\log x)+O_{f,\varepsilon}(x^{\theta_l+\varepsilon})$, where $\tilde{P}_2(t), \tilde{P}_4(t),\tilde{P}_6(t)$ and $\tilde{P}_8(t)$ are polynomials of degree 0, 1, 4 and 13, respectively, and $\tilde{P}_l(t)\equiv 0$ for $l=3,5,7$ and $\theta_2=\frac{8}{11}, \theta_3=\frac{17}{20}, \theta_4=\frac{43}{46},\theta_5=\frac{83}{86}, \theta_6=\frac{184}{187}, \theta_7=\frac{355}{358}, \theta_8=\frac{752}{755}$. Recently, Xu \cite{Xu} has refined and generalized the above work of Zu for all integers $l\geq 2$ using the recent celebrated work of Newton and Thorne \cite{Newton and Thorne2021, Newton and Thorne2021 II}.

Considering the coefficients $\lambda_{{\rm{sym}}^2  f}(n)$ of the symmetric square $L$-function $L(s, {{\rm{sym}}^2  f})$, Fomenko \cite{Fomenko2006, Fomenko2008}  studied the sums 
\begin{align*}
    &\sum_{n\leq x}\lambda_{{\rm{sym}}^2  f}(n)& \text{ and }&
    &\sum_{n\leq x}\lambda_{{\rm{sym}}^2  f}^2(n).
\end{align*}
Later, these sums have been studied and generalized by many authors; see \cite{He, Jiang and Lu, Luo Lao and Zou, AS Singh and Srinivas, Tang and Wu}. 

Recently, Sharma and Sankaranarayanan \cite{ASAS2022} studied the sum
\begin{equation}
   U_{f,j}(x):= \sum_{\substack{n=a_1^2+a_2^2+a_3^2+a_4^2\leq x\\a_1,a_2,a_3, a_4\in \mathbb{Z}}}\lambda_{{\rm{sym}}^j  f}^2(n)
\end{equation}
for $j=2$ and they established that 
\begin{equation*}
     U_{f,2}(x)=C_{f,2}x^2+O(x^{\frac{9}{5}+\varepsilon}).
\end{equation*}
Later, Hua \cite{Hua} improved and generalized the work of Sharma and Sankaranarayanan; in fact, he established that 
\begin{equation}
    U_{f,j}(x)=\mathcal{C}_{f, j} x^2+O\left(x^{2-\frac{60}{30(j+1)^2-13}+\varepsilon}\right),\label{E3}
\end{equation}
for all integers $j\geq 2$ and for some effective constant $\mathcal{C}_{f,j}$.

In another work, Sharma and Sankaranarayanan \cite{ASAS2023} considered the sum 
\begin{equation}
   V_{f,j}(x):= \sum_{\substack{n=a_1^2+a_2^2+\cdots+a_6^2\leq x\\a_1,a_2,\cdots, a_6\in \mathbb{Z}}}\lambda_{{\rm{sym}}^j  f}^2(n)
\end{equation}
and proved that 
\begin{equation}
     V_{f,j}(x)=C_{f,j}'x^3+O\left( x^{3-\frac{6}{3(j+1)^2+1}+\varepsilon}\right)\label{E4}
\end{equation}
 for all integers $j\geq 2$ and for some effective constant $\mathcal{C}_{f,j}'$.

Recently, Liu and Yang \cite{Liu and Yang} improved the error term bounds in \eqref{E3} and \eqref{E4}, and the improved bounds are $O\left(x^{2-\frac{120}{60(j+1)^2-61}+\varepsilon}\right)$ and $O\left( x^{3-\frac{210}{105(j+1)^2-103}+\varepsilon}\right)$, respectively for all integers $j\geq 2$. Very recently, Feng in \cite{Feng}, further improved the above estimates of Liu and Yang, and the improved asymptotic formulae are 
\begin{align}
     &U_{f,j}(x)=\mathcal{C}_{f, j} x^2+O\left(x^{2-\frac{10}{10k_j+12+5(j-1)(j+3)}+\varepsilon}\right),\label{E5}\\
     &V_{f,j}(x)=C_{f,j}'x^3+O\left( x^{3-\frac{10}{10k_j+12+5(j-1)(j+3)}+\varepsilon}\right),\label{E6}
\end{align}
 for $j\geq 3$, where $k_3=11/40, k_4=5/26, k_5=23/130$ and  $k_j=\frac{8}{63}\sqrt{15/(2j-1)}$ for all integers $j\geq 6$.

The purpose of this paper is to further improve the results \eqref{E5} and \eqref{E6} of Feng. 

Precisely, we establish:
\begin{theorem}
Let $f\in H_\kappa$ and $j\geq 3$. Then we have 
\begin{equation*}
   U_{f,j}(x)=\mathcal{C}_{f, j} x^2+O\left(x^{2-\frac{630j^\frac{3}{2}}{315j^\frac{3}{2}(j+1)^2-504j^\frac{3}{2}+80\sqrt{15}}+\varepsilon}\right),
\end{equation*}
for some effective constant $\mathcal{C}_{f,j}$.\label{T1.1}
\end{theorem}

\begin{theorem}
Let $f\in H_\kappa$ and $j\geq 3$. Then we have 
\begin{equation*}
    V_{f,j}(x)=\mathcal{C}_{f, j}' x^3+O\left(x^{3-\frac{630j^\frac{3}{2}}{315j^\frac{3}{2}(j+1)^2-504j^\frac{3}{2}+80\sqrt{15}}+\varepsilon}\right),
\end{equation*}
for some effective constant $\mathcal{C}_{f,j}'$.\label{T1.2}
\end{theorem}
Moreover, we further improve these results for $j\geq 127$, and precisely we prove:
\begin{theorem}
Let $f\in H_\kappa$ and $j\geq 127$. Then we have 
\begin{equation*}
   U_{f,j}(x)=\mathcal{C}_{f, j} x^2+O\left( x^{2-\frac{126j^\frac{1}{4}}{63j^\frac{1}{4}(j+1)^2+63j^\frac{3}{4}-378j^\frac{1}{4}+16\sqrt{15}}+\varepsilon}\right),
\end{equation*}
for some effective constant $\mathcal{C}_{f,j}$.\label{T1.3}
\end{theorem}

\begin{theorem}
Let $f\in H_\kappa$ and $j\geq 127$. Then we have 
\begin{equation*}
    V_{f,j}(x)=\mathcal{C}_{f, j}' x^3+O\left(x^{3-\frac{126j^\frac{1}{4}}{63j^\frac{1}{4}(j+1)^2+63j^\frac{3}{4}-378j^\frac{1}{4}+16\sqrt{15}}+\varepsilon}\right),
\end{equation*}
for some effective constant $\mathcal{C}_{f,j}'$.\label{T1.4}
\end{theorem}

\begin{remark}
    Note that 
    \begin{equation*}
        \frac{10}{10k_j+12+5(j-1)(j+3)}<\frac{630j^\frac{3}{2}}{315j^\frac{3}{2}(j+1)^2-504j^\frac{3}{2}+80\sqrt{15}}
    \end{equation*}
    and 
    \begin{equation*}
        \frac{630j^\frac{3}{2}}{315j^\frac{3}{2}(j+1)^2-504j^\frac{3}{2}+80\sqrt{15}}<\frac{126j^\frac{1}{4}}{63j^\frac{1}{4}(j+1)^2+63j^\frac{3}{4}-378j^\frac{1}{4}+16\sqrt{15}}
    \end{equation*}
    for $j\geq 3$. Thus, Theorems \ref{T1.1} to \ref{T1.4} improve upon the earlier results of Feng \cite{Feng}. Moreover, it is not difficult to further refine the error term bounds in Theorems \ref{T1.1} and \ref{T1.2} by moving the line of integration to $\Re(s)=2-\sigma(j)$ with $0<\sigma(j)<\frac{1}{j^3}$ and applying the same arguments as in our proofs for $j\geq 3$. For example, $\sigma(j)=\frac{1}{j^4}, \frac{1}{j^5}, \cdots$.    
    Similarly, the error term bounds in Theorems \ref{T1.3} and \ref{T1.4} can be improved by moving the line of integration to $\Re(s)=2-\sigma^*(j)$ with $\sigma*(j)>\frac{1}{\sqrt{j}}$ and following the same arguments of our theorems. However, in these cases, the improvements occur only for large values of $j$. For example, if $\sigma^*(j)=\frac{1}{j^\frac{1}{3}}$, improvement holds for $j\geq 1425$, if $\sigma^*(j)=\frac{1}{j^\frac{1}{4}}$, the improvement holds for $j\geq 16035$. 
\end{remark}

\section{Lemmas}
Here, we state some lemmas, which we use in the proofs of the main theorems. Let $r_k(n)=\#\{(n_1,n_2,\cdots, n_k)\in \mathbb{Z}^k:n_1^2+n_2^2+\cdots+n_k^2=n\}$ allowing zeros,
distinguishing signs and order. We are interested in the two functions $r_4(n)$ and $r_6(n)$. 
\begin{lemma}
    For any positive integer $n$, we have 
    \begin{equation*}
        r_4(n)=8\sum_{d|n, 4\nmid n}d.
    \end{equation*}\label{L2.1}
\end{lemma}
\begin{proof}
    See \cite{Hardy and Wright}.
\end{proof}
We can write $\displaystyle r_4(n)=8\sum_{d|n}\widetilde{\chi}_0(d)d$, where $\widetilde{\chi}_0$ is a character modulo $4$ given by 
\begin{equation}
    \widetilde{\chi}_0(p^v):=\begin{cases}
        \chi_0(p^v)&\text{ if $p>2$}\\
        3&\text{ if $p=2$}
    \end{cases}\label{E7}
\end{equation}
and $\chi_0$ is the principal character modulo 4. We write $\displaystyle r(n):=\sum_{d|n}\widetilde{\chi}_0(d)d$, which is multiplicative and is given by 
\begin{equation*}
    r(p^v)=\begin{cases}
        \frac{1-p^{v+1}}{1-p}&\text{ if $p>2$}\\
        3&\text{ if $p=2$}.
    \end{cases}
\end{equation*}
From the above information, we note that 
\begin{align}
     U_{f,j}(x)&=\sum_{\substack{n=a_1^2+a_2^2+a_3^2+a_4^2\leq x\\a_1,a_2,a_3, a_4\in \mathbb{Z}}}\lambda_{{\rm{sym}}^j  f}^2(n)\nonumber\\
     &=\sum_{n\leq x}\lambda_{{\rm{sym}}^j  f}^2(n)\sum_{\substack{n=a_1^2+a_2^2+a_3^2+a_4^2\leq x\\a_1,a_2,a_3, a_4\in \mathbb{Z}}}1\nonumber\\
     &=\sum_{n\leq x}\lambda_{{\rm{sym}}^j  f}^2(n)r_4(n)\nonumber\\
     &=8\sum_{n\leq x}\lambda_{{\rm{sym}}^j  f}^2(n)r(n),\label{E8}
\end{align}
where  $\displaystyle r(n)=\sum_{d|n}\widetilde{\chi}_0(d)d$ and  $\displaystyle r(p)=\sum_{d|p}\widetilde{\chi}_0(d)d=1+p\widetilde{\chi}_0(p)$.

\begin{lemma}
    For any positive integer $n$, we have 
    \begin{equation*}
        r_6(n)=16\sum_{d|n}\chi(d)\frac{n^2}{d^2}-4\sum_{d|n}\chi(d)d^2,
        \end{equation*}
        where $\chi$ is the nonprincipal Dirichlet character modulo $4$
\begin{equation}
    \chi(n)=\begin{cases}
        1&\text{if $n\equiv 1\pmod 4$},\\
        -1&\text{if $n\equiv -1\pmod 4$},\\
        0&\text{if $n\equiv 0\pmod 2$}.
    \end{cases}\label{E9}
\end{equation}\label{L2.2}
\end{lemma}
\begin{proof}
    See \cite{Hardy and Wright}.
\end{proof}
We write $\displaystyle r_6(n)=16l(n)-4v(n)$, where $\displaystyle l(n)=\sum_{d|n}\chi(d)\frac{n^2}{d^2}$ and $\displaystyle v(n)=\sum_{d|n}\chi(d)d^2$. We note that $\chi(d)$ and $\frac{n^2}{d^2}$ are multiplicative. Thus, following Theorem 265 of \cite{Hardy and Wright}, we find that both the functions $l(n)$ and $v(n)$ are multiplicative. 

Hence, we can write 
\begin{align}
     V_{f,j}(x)&= \sum_{\substack{n=a_1^2+a_2^2+\cdots+a_6^2\leq x\\a_1,a_2,\cdots, a_6\in \mathbb{Z}}}\lambda_{{\rm{sym}}^j  f}^2(n)\nonumber\\
     &=\sum_{n\leq x}\lambda_{{\rm{sym}}^j  f}^2(n)\sum_{\substack{n=a_1^2+a_2^2+\cdots+a_6^2\leq x\\a_1,a_2,\cdots, a_6\in \mathbb{Z}}}1\nonumber\\
     &=\sum_{n\leq x}\lambda_{{\rm{sym}}^j  f}^2(n)r_6(n)\nonumber\\
     &=16\sum_{n\leq x}\lambda_{{\rm{sym}}^j  f}^2(n)l(n)-4\sum_{n\leq x}\lambda_{{\rm{sym}}^j  f}^2(n)v(n),\label{E10}
\end{align}
where $\displaystyle l(n)=\sum_{d|n}\chi(d)\frac{n^2}{d^2}$ and $\displaystyle v(n)=\sum_{d|n}\chi(d)d^2$. Note that $l(p)=p^2+\chi(p) $ and $v(p)=1+p^2\chi(p)$.
\begin{lemma}
    For $j\geq 3$ and $f\in H_\kappa$, we have 
    \begin{equation*}
        \mathcal{F}_{1,j}(s):=\sum_{n=1}^\infty\frac{\lambda_{{\rm{sym}}^j  f}^2(n)r(n)}{n^s}=\mathcal{G}_{1,j}(s)\mathcal{H}_{1,j}(s),
    \end{equation*}
    where 
    \begin{equation*}
        \mathcal{G}_{1,j}(s)=\zeta(s)L(s-1,\widetilde{\chi}_0)\prod_{n=1}^jL(s, {\rm{sym}}^{2n}  f)L(s-1, {\rm{sym}}^{2n}  f\otimes\widetilde{\chi}_0),
    \end{equation*}
    $\widetilde{\chi}_0$ is the character as in \eqref{E7} and $\mathcal{H}_{1,j}(s)$ is some Dirichlet series which converges absolutely in $\Re(s)\geq \frac{3}{2}+\varepsilon$ and $\mathcal{H}_{1,j}(2)\neq 0$.\label{L2.3}
\end{lemma}
\begin{proof}
    See \cite{Hua}.
\end{proof}

\begin{lemma}
    For $j\geq 3$ and $f\in H_\kappa$, we have 
     \begin{equation*}
        \mathcal{F}_{2,j}(s):=\sum_{n=1}^\infty\frac{\lambda_{{\rm{sym}}^j  f}^2(n)l(n)}{n^s}=\mathcal{G}_{2,j}(s)\mathcal{H}_{2,j}(s),
    \end{equation*}
    where 
     \begin{equation*}
        \mathcal{G}_{2,j}(s)=\zeta(s-2)L(s,\chi)\prod_{n=1}^jL(s-2, {\rm{sym}}^{2n}  f)L(s, {\rm{sym}}^{2n}  f\otimes\chi),
    \end{equation*}
    $\chi$ is the character as in \eqref{E8} and $\mathcal{H}_{2,j}(s)$ is some Dirichlet series which converges absolutely in $\Re(s)\geq \frac{5}{2}+\varepsilon$ and $\mathcal{H}_{2,j}(3)\neq 0$.\label{L2.4}
\end{lemma}
\begin{proof}
    See \cite{ASAS2023}.
\end{proof}

\begin{lemma}
    For $j\geq 3$ and $f\in H_\kappa$, we have 
     \begin{equation*}
        \mathcal{F}_{3,j}(s):=\sum_{n=1}^\infty\frac{\lambda_{{\rm{sym}}^j  f}^2(n)v(n)}{n^s}=\mathcal{G}_{3,j}(s)\mathcal{H}_{3,j}(s),
    \end{equation*}
    where 
     \begin{equation*}
        \mathcal{G}_{3,j}(s)=\zeta(s)L(s-2,\chi)\prod_{n=1}^jL(s, {\rm{sym}}^{2n}  f)L(s-2, {\rm{sym}}^{2n}  f\otimes\chi),
    \end{equation*}
    $\chi$ is the character as in \eqref{E8} and $\mathcal{H}_{3,j}(s)$ is some Dirichlet series which converges absolutely in $\Re(s)\geq \frac{5}{2}+\varepsilon$ and $\mathcal{H}_{3,j}(3)\neq 0$.\label{L2.5}
\end{lemma}
\begin{proof}
    See \cite{ASAS2023}.
\end{proof}

\begin{lemma}
For $f\in H_\kappa$ and $i\geq 0$, we have
\begin{equation*}
    L(s-1,{\rm{sym}}^{i}  f\otimes\widetilde{\chi}_o)= \left(1-\frac{3\lambda_{{\rm{sym}}^{i}f}(2)}{2^{s-1}}\right)^{-1}\left(1-\frac{\lambda_{{\rm{sym}}^{i}f}(2)}{2^{s-1}}\right)^2 L(s-1,{\rm{sym}}^{i}  f).
\end{equation*}\label{L2.6}
\end{lemma}
\begin{proof}
By definition, we have
\begin{align*}
	& L(s-1,{\rm{sym}}^{i}  f\otimes\widetilde{\chi}_o)\\
     =&\sum_{n=1}^\infty\frac{\lambda_{{\rm{sym}}^{i}f}(n)\widetilde{\chi}_o(n)}{n^{s-1}}\\
		=&\prod_p\left(1-\frac{\lambda_{{\rm{sym}}^{i}f}(p)\widetilde{\chi}_o(p)}{p^{s-1}}\right)^{-1}\\
        =&\left(1-\frac{3\lambda_{{\rm{sym}}^{i}f}(2)}{2^{s-1}}\right)^{-1}\prod_{p>2}\left(1-\frac{\lambda_{{\rm{sym}}^{i}f}(p)\chi_o(p)}{p^{s-1}}\right)^{-1}\\
        =&\left(1-\frac{3\lambda_{{\rm{sym}}^{i}f}(2)}{2^{s-1}}\right)^{-1}\left(1-\frac{\lambda_{{\rm{sym}}^{i}f}(2)}{2^{s-1}}\right)\prod_{p}\left(1-\frac{\lambda_{{\rm{sym}}^{i}f}(p)\chi_o(p)}{p^{s-1}}\right)^{-1}\\
        =&\left(1-\frac{3\lambda_{{\rm{sym}}^{i}f}(2)}{2^{s-1}}\right)^{-1}\left(1-\frac{\lambda_{{\rm{sym}}^{i}f}(2)}{2^{s-1}}\right)\prod_{\substack{p\\(p,4)=1}}\left(1-\frac{\lambda_{{\rm{sym}}^{i}f}(p)}{p^{s-1}}\right)^{-1}\\
        =&\left(1-\frac{3\lambda_{{\rm{sym}}^{i}f}(2)}{2^{s-1}}\right)^{-1}\left(1-\frac{\lambda_{{\rm{sym}}^{i}f}(2)}{2^{s-1}}\right)^2\prod_{p}\left(1-\frac{\lambda_{{\rm{sym}}^{i}f}(p)}{p^{s-1}}\right)^{-1}\\
        =&\left(1-\frac{3\lambda_{{\rm{sym}}^{i}f}(2)}{2^{s-1}}\right)^{-1}\left(1-\frac{\lambda_{{\rm{sym}}^{i}f}(2)}{2^{s-1}}\right)^2 L(s-1,{\rm{sym}}^{i}  f).
\end{align*}
 Note that $ L(s,{\rm{sym}}^{i}  f)=\zeta(s)$ when $i=0$. So, in the particular case when $i=0$, we have $\lambda_{{\rm{sym}}^{i}f}(2)=1$ and 
 \begin{equation*}
    L(s-1,\widetilde{\chi}_o)=  \left(1-\frac{3\lambda_{{\rm{sym}}^{i}f}(2)}{2^{s-1}}\right)^{-1}\left(1-\frac{\lambda_{{\rm{sym}}^{i}f}(2)}{2^{s-1}}\right)^2 \zeta(s-1).
\end{equation*}
\end{proof}
Note that similar equalities hold as in Lemma \ref{L2.6}, even when $\widetilde{\chi}_o$ is replaced with the Dirichlet character $\chi$, which is in equation \eqref{E9}.
\begin{lemma}
Suppose that $\mathfrak{L}(s)$ is a general $L$-function of degree $m$. Then for any $\epsilon>0$, we have
		\begin{equation}
			\int_T^{2T}\mid\mathfrak{L}(\sigma+it)\mid^2dt\ll T^{\max\{m(1-\sigma),\ 1\}+\epsilon}\label{E11}
		\end{equation}
		uniformly for $\frac{1}{2}\leq \sigma\leq 1$ and $T>1$; and 
		\begin{equation}
			\mathfrak{L}(\sigma+it)\ll (10+ \mid t\mid)^{\frac{m}{2}(1-\sigma)+\epsilon}\label{E12}
		\end{equation}
		uniformly for $\frac{1}{2}\leq \sigma\leq 1+\epsilon$ and $\mid t\mid>10$.\label{L2.7}
	\end{lemma}
	
	\begin{proof}
		The result \eqref{E9} follows from Perelli \cite{Perelli}, and \eqref{E10} follows from the maximum modulus principle.
	\end{proof}

\begin{lemma}
    Let $K=\frac{8\sqrt{15}}{63}$. Then for $\varepsilon>0$, we have
    \begin{equation}
        \zeta(\sigma+it)\ll |t|^{K(1-\sigma)^\frac{3}{2}+\varepsilon}\label{E13}
    \end{equation}
    uniformly for $|t|\geq 10$ and $\frac{1}{2}\leq\sigma\leq 1$.\label{L2.8}
\end{lemma}
\begin{proof}
    The result is due to Heath-Brown. See \cite{Heath-Brown}.
\end{proof}

\begin{lemma}
		For $\frac{1}{2}\leq \sigma\leq 2$, $T$ sufficiently large, there exists a $T^*\in[T,T+T^\frac{1}{3}]$ such that 
		\begin{equation*}
			\log\zeta(\sigma+iT^*)\ll (\log\log T^*)^2\ll(\log\log T)^2
		\end{equation*}
		holds. Thus we have 
		\begin{equation}
			\mid \zeta(\sigma+it)\mid \ll \exp((\log\log T^*)^2)\ll T^\epsilon\label{E14}
		\end{equation}
		on the horizontal line with $t=T^*$ uniformly for $\frac{1}{2}\leq \sigma\leq 2$.\label{L2.9}
	\end{lemma}
	\begin{proof}
		See Lemma $1$ of \cite{KRAS}.
	\end{proof}

\begin{lemma}
    For any $\varepsilon>0$, we have
    \begin{equation}
			L(\sigma+it,\ {\rm{sym}}^2f)\ll (10+\mid t\mid)^{\max \{\frac{6}{5}(1-\sigma), 0\}+\epsilon} \label{E15}
		\end{equation}
        holds uniformly for $\frac{1}{2}\leq \sigma\leq 1+\varepsilon$ and $\mid t\mid\geq 10$; and 
        \begin{equation}
            \int_1^T|L(\sigma'+it, {\rm{sym}^2f})|^\frac{12772}{5251}dt\ll T^{1+\varepsilon},\label{E16}
        \end{equation}
        uniformly for $|T|\geq 10$ and $\sigma'=\frac{27133}{38316}$.
        \label{L2.10}
\end{lemma}
\begin{proof}
   For the subconvexity bound in \eqref{E15} see \cite{LinNunesQi} and for \eqref{E16} see \cite{Wang}.
\end{proof}

\begin{lemma}
   Let $\mathfrak{L}(s,f)$ be an $L$-function of degree $m$. Then for any $\varepsilon>0$ and character $\chi$, we have
    \begin{align}
       & L(\sigma+it, \chi)\ll (10+\mid t\mid )^{\max \{\frac{13}{42}(1-\sigma), 0\}+\epsilon}\label{E17}\\
       & L(\sigma+it, {\rm{sym}}^2f\otimes \chi)\ll (10+\mid t\mid)^{\max \{\frac{6}{5}(1-\sigma), 0\}+\epsilon} \label{E18}\\
       &\mathfrak{L}(\sigma+iT, f\otimes \chi)\ll ( 10+\mid t\mid)^{\frac{m}{2}(1-\sigma)+\epsilon}\label{E19} \\
       &\int_1^{T}\mid\mathfrak{L}(\sigma+it,f\otimes \chi)\mid^2 dt\ll T^{\max\{m(1-\sigma),\ 1\}+\epsilon}\label{E20}
    \end{align}
   holds uniformly for $\frac{1}{2}\leq \sigma\leq 1+\varepsilon$ and $\mid t\mid\geq 10$.\label{L2.11}
\end{lemma}

\begin{proof}
    For a general $L$-function $\mathfrak{L}(s,f)$, the corresponding twisted $L$-function $\mathfrak{L}(s,f\otimes \chi)$ is also a general $L$-function of the same degree in the sense of Perelli \cite{Perelli}. Thus, twisting by a character does not affect the convexity, subconvexity, and integral mean value estimates of an $L$-function. In \cite{Huang}, Huang handled $SL(3)$ $L$-functions twisted by a quadratic primitive character with large modulus, and in \cite{Liu}, Liu gave a similar proof for \eqref{E17}. The equations from \eqref{E18} to \eqref{E20} follow similar to \eqref{E17} from \cite{Liu}.
\end{proof}

\begin{lemma}
    Let $\lambda>0,\ \mu>0$ and $\alpha<\sigma< \beta$. Then we have 
\begin{align*}
    J(\sigma, p\lambda+q\mu)=O\{J^p(\alpha, \lambda)J^q(\beta, \mu)\},\label{E7}
\end{align*}
where $\displaystyle J(\sigma, \lambda)=\left\{\int_0^T| f(\sigma+it)|^\frac{1}{\lambda}dt \right\}^\lambda$, $p=\frac{\beta-\sigma}{\beta-\alpha}$ and $q=\frac{\sigma-\alpha}{\beta-\alpha}$.\label{L2.12}
\end{lemma}
\begin{proof}
    See pp. 236 of \cite{Titchmarsh}.
\end{proof}

\begin{lemma}
    Let $j\geq 3$. Then for any $\varepsilon>0$, we have
    \begin{equation*}
        \left(\int_{10}^T|L(\sigma+it,{\rm{sym}}^{2j}f)|^\frac{12772}{1135}dt\right)^{\frac{1135}{12772}}\ll T^{\frac{1135}{12772}+\frac{2j-1}{2}(1-\sigma)+\varepsilon},
    \end{equation*}
    uniformly for $T\geq 10$ and $\frac{11637}{12772}< \sigma< 1$.\label{L2.13}
\end{lemma}

\begin{proof}
We prove this lemma using Lemma \ref{L2.12}. For that, we choose the parameters accordingly, as $\alpha=\frac{1}{2}$, $\beta=1$, and $\lambda=\frac{1}{2}$. Then, we have $p=2(1-\sigma)$ and $q=1-2(1-\sigma)$. We let $p\lambda+q\mu=\frac{1135}{12772}$, which implies that $\mu= \frac{1}{2\sigma-1}\left(\frac{1135}{12772}-(1-\sigma)\right)$. Note that $\mu$ is positive since $\sigma> \frac{11637}{12772}$. Now,  following the Lemma \ref{L2.12}, we get 
\begin{align*}
     & \left(\int_{10}^T|L(\sigma+it,{\rm{sym}}^{2j}f)|^\frac{12772}{1135}dt\right)^{\frac{1135}{12772}}\\
     &\ll \left(\int_{10}^T|L(\frac{1}{2}+it,{\rm{sym}}^{2j}f)|^2dt\right)^{p\lambda}\left(\int_{10}^T|L(1+it,{\rm{sym}}^{2j}f)|^\frac{1}{\mu}dt\right)^{q\mu }\\
     &\ll_\varepsilon T^{\frac{2j+1}{2}\frac{2(1-\sigma)}{2}+\frac{1135}{12772}-(1-\sigma)+\varepsilon}\\
      &\ll_\varepsilon T^{\frac{1135}{12772}+\frac{2j-1}{2}(1-\sigma)+\varepsilon},
\end{align*}
which follows from Lemma \ref{L2.7}.
\end{proof}

\begin{lemma}
    Let $j\geq 127$. Then for any $\varepsilon>0$, we have
   \begin{equation}
            \int_1^T|L(1-\frac{1}{\sqrt{j}}+it, {\rm{sym}^2f})|^\frac{12772}{5251}dt\ll T^{1+\varepsilon},\label{E21}
        \end{equation}
    uniformly for $T\geq 10$.\label{L2.14} 
\end{lemma}

\begin{proof}
    Follows similar to the Lemma \ref{L2.13}, from \eqref{E16} and the Lemma \ref{L2.12} by choosing the parameters as $\alpha=\frac{27133}{38316}$, $\sigma=1-\frac{1}{\sqrt{j}}$, $\beta=1$, $\lambda=\frac{5251}{12772}$, and $p\lambda+q\mu=\frac{5251}{12772}$. Note that for this set of parameters, we have 
    \begin{equation*}
        \mu=\frac{1}{q}\left(\frac{5251}{12772}-\frac{5251}{12772}\frac{38316}{11183\sqrt{j}}\right)>0
    \end{equation*}
     since $j\geq 127$. 
\end{proof}

\section{Proof of Theorem \ref{T1.1}} \label{Sec3}
We apply Perron's formula to $\mathcal{F}_{1,j}(s)$, then, following Lemma \ref{L2.3}, we have 
\begin{align*}
    8 \sum_{n\leq x}\lambda_{{\rm{sym}}^j  f}^2(n) r(n)= \frac{8}{2\pi i}\int_{2+\varepsilon-iT}^{2+\varepsilon+iT}\mathcal{F}_{1,j}(s)\frac{x^s}{s}ds+O\left(\frac{x^{2+\varepsilon}}{T}\right), 
\end{align*}
where $10\leq T\leq x$ is a parameter to be chosen later, and $\mathcal{F}_{1,j}(s)$ is as in Lemma \ref{L2.3}.

Now, we move the line of integration to $\Re(s)=2-\frac{1}{j^3}$. Note that in the rectangle $\mathcal{R}$ formed by the line segments joining the points $2+\varepsilon-iT,\ 2+\varepsilon+iT,\ 2-\frac{1}{j^3}+iT$ and $2-\frac{1}{j^3}-iT$, $\mathcal{F}_{1,j}(s)$ is a meromorphic function having a simple pole at $s=2$, which arises from the factor $\zeta(s-1)$ in the decomposition of $\mathcal{F}_{1,j}(s)$. Thus, Cauchy's residue theorem implies 
\begin{align*}
   & 8\sum_{n\leq x}\lambda_{{\rm{sym}}^j  f}^2(n) r(n)\\
   &=\mathcal{C}_{f, j} x^2+\frac{8}{2\pi i}\left\{\int_{2+\varepsilon-iT}^{2-\frac{1}{j^3}-iT}+\int_{2-\frac{1}{j^3}-iT}^{2-\frac{1}{j^3}+iT}+\int_{2-\frac{1}{j^3}+iT}^{2+\varepsilon+iT}\right\}\mathcal{F}_{1,j}(s)\frac{x^s}{s}ds\\
   &\qquad\qquad+O\left(\frac{x^{2+\varepsilon}}{T}\right)\\
    &:=\mathcal{C}_{f, j} x^2+I_1+I_2+I_3+O\left(\frac{x^{2+\varepsilon}}{T}\right),
\end{align*}
where $\displaystyle\mathcal{C}_{f, j} x^2=8\mathop{\mathrm{Res}}_{s=2}\mathcal{F}_{1,j}(s)\frac{x^s}{s}$.

Here we make the special choice $T=T^*$ of Lemma \ref{L2.9}, which satisfies \eqref{E14}, so that the horizontal portions $I_2$ and $I_3$ are controlled by the vertical line contribution $I_1$. The contribution of $I_1$ is given by
\begin{align*}
    I_1&\ll \int_{2-\frac{1}{j^3}-iT}^{2-\frac{1}{j^3}+iT} \left|\zeta(s)\zeta(s-1)\prod_{n=1}^j L(s, {\rm{sym}^{2n}}f)L(s-1, {\rm{sym}^{2n}}f)\right|\frac{x^s}{s}ds\\
    &\ll x^{2-\frac{1}{j^3}+\varepsilon}+x^{2-\frac{1}{j^3}+\varepsilon} \int_{10}^T\left|\zeta(1-\frac{1}{j^3}+it)\prod_{n=1}^j L(1-\frac{1}{j^3}+it, {\rm{sym}^{2n}}f)\right| t^{-1}dt\\
    &\ll x^{2-\frac{1}{j^3}+\varepsilon}+x^{2-\frac{1}{j^3}+\varepsilon} \sup_{10\leq T_1\leq T}\left(\int_{T_1}^{2T_1}\mid L(1-\frac{1}{j^3}+it,{\rm{sym}}^4f)\mid^{2}dt\right)^\frac{1}{2}\\
    &\qquad\left(\int_{T_1}^{2T_1}\mid \prod_{n=3}^j L(1-\frac{1}{j^3}+it, {\rm{sym}^{2n}}f)\mid^{2}dt\right)^\frac{1}{2} \times\\
    &\qquad \left\{\max_{T_1\leq t\leq 2T_1} \zeta(1-\frac{1}{j^3}+it)\mid L(1-\frac{1}{j^3}+it,{\rm{sym}}^2f)\mid\right\}T_1^{-1}\\
    &\ll x^{2-\frac{1}{j^3}+\varepsilon} T^{\frac{8\sqrt{15}}{63}\times\left(\frac{1}{j^3}\right)^\frac{3}{2}+\frac{6}{5}\times \frac{1}{j^3}+\frac{5}{2}\times \frac{1}{j^3}+\frac{(j+1)^2-9}{2}\times \frac{1}{j^3}-1+\varepsilon}\\
    &\ll x^{2-\frac{1}{j^3}+\varepsilon}T^{\frac{8\sqrt{15}}{63}\times\frac{1}{j^\frac{9}{2}}+\{\frac{(j+1)^2}{2}-\frac{4}{5} \}\frac{1}{j^3}-1+\varepsilon},
\end{align*}
which follows from Lemmas \ref{L2.7}, \ref{L2.8} and \ref{L2.10}. 

Now, the contributions of $I_2$ and $I_3$ are given by
\begin{align*}
    |I_2|+|I_3|&\ll \int_{2-\frac{1}{j^3}+iT}^{2+\varepsilon+iT}\left|\zeta(s)\zeta(s-1)\prod_{n=1}^j L(s, {\rm{sym}^{2n}}f)L(s-1, {\rm{sym}^{2n}}f)\right|\frac{x^s}{s}ds\\
    &\ll \int_{1-\frac{1}{j^3}}^{1+\varepsilon} \left|\zeta(\sigma+iT)\prod_{n=1}^j L(\sigma+iT, {\rm{sym}^{2n}}f)\right|x^{1+\sigma}T^{-1}d\sigma\\
    &\ll \int_{1-\frac{1}{j^3}}^{1+\varepsilon} x^{1+\sigma} T^{\varepsilon+\frac{6}{5}(1-\sigma)+\frac{(j+1)^2-4}{2}(1-\sigma)-1}d\sigma\\
    &\ll xT^{\frac{(j+1)^2}{2}-\frac{4}{5}-1+\varepsilon} \int_{1-\frac{1}{j^3}}^{1+\varepsilon}\left(\frac{x}{T^{\frac{(j+1)^2}{2}-\frac{4}{5}}}\right)^{\sigma}d\sigma.
\end{align*}

For $j\geq 3$ and $10\leq T\leq x$, note that $\left(\frac{x}{T^{\frac{(j+1)^2}{2}-\frac{4}{5}}}\right)^{\sigma}$ is monotonic as a function of $\sigma$ in the interval $[1-\frac{1}{j^3},1+\varepsilon]$ and thus the maximum attains at the boundary points. Hence,
\begin{align*}
    |I_2|+|I_3| &\ll xT^{\frac{(j+1)^2}{2}-\frac{4}{5}-1+\varepsilon} \max_{1-\frac{1}{j^3}\leq \sigma\leq 1+\varepsilon}\left(\frac{x}{T^{\frac{(j+1)^2}{2}-\frac{4}{5}}}\right)^{\sigma}\\
    &\ll \frac{x^{2+\varepsilon}}{T}+x^{2-\frac{1}{j^3}}T^{\frac{(j+1)^2}{2}\frac{1}{j^3}-\frac{4}{5}\frac{1}{j^3}-1+\varepsilon}.
\end{align*}

Therefore, in total, we have
\begin{align*}
    8\sum_{n\leq x}\lambda_{{\rm{sym}}^j  f}^2(n) r(n)=&\mathcal{C}_{f, j} x^2+O\left(x^{2-\frac{1}{j^3}+\varepsilon}T^{\frac{8\sqrt{15}}{63}\times\frac{1}{j^\frac{9}{2}}+\{\frac{(j+1)^2}{2}-\frac{4}{5} \}\frac{1}{j^3}-1+\varepsilon}\right)\\
    &+O\left(\frac{x^{2+\varepsilon}}{T}\right).
\end{align*}
Finally, making our choice of $T$ as $x^{2-\frac{1}{j^3}}T^{\frac{8\sqrt{15}}{63}\times\frac{1}{j^\frac{9}{2}}+\{\frac{(j+1)^2}{2}-\frac{4}{5} \}\frac{1}{j^3}-1}=\frac{x^2}{T}$, i.e.,  $T=x^{\frac{630j^\frac{3}{2}}{315j^\frac{3}{2}(j+1)^2-504j^\frac{3}{2}+80\sqrt{15}}}$, we obtain
\begin{equation*}
    8\sum_{n\leq x}\lambda_{{\rm{sym}}^j  f}^2(n) r(n)=\mathcal{C}_{f, j} x^2+O\left(x^{2-\frac{630j^\frac{3}{2}}{315j^\frac{3}{2}(j+1)^2-504j^\frac{3}{2}+80\sqrt{15}}}\right).
\end{equation*}
This completes the proof of Theorem \ref{T1.1}.

\section{Proof of Theorem \ref{T1.2}}\label{Sec4}
We apply Perron's formula to $\mathcal{F}_{2,j}(s)$, then, following Lemma \ref{L2.4}, we have 
\begin{align*}
     16\sum_{n\leq x}\lambda_{{\rm{sym}}^j  f}^2(n) l(n)= \frac{16}{2\pi i}\int_{3+\varepsilon-iT}^{3+\varepsilon+iT}\mathcal{F}_{2,j}(s)\frac{x^s}{s}ds+O\left(\frac{x^{3+\varepsilon}}{T}\right), 
\end{align*}
where $10\leq T\leq x$ is a parameter to be chosen later, and $\mathcal{F}_{2,j}(s)$ is as in Lemma \ref{L2.4}.

Now, we move the line of integration to $\Re(s)=3-\frac{1}{j^3}$. Note that in the rectangle $\mathcal{R}^*$ formed by the line segments joining the points $3+\varepsilon-iT,\ 3+\varepsilon+iT,\ 3-\frac{1}{j^3}+iT$ and $3-\frac{1}{j^3}-iT$, $\mathcal{F}_{2,j}(s)$ is a meromorphic function having a simple pole at $s=3$, which arises from the factor $\zeta(s-2)$ in the decomposition of $\mathcal{F}_{2,j}(s)$. Thus, Cauchy's residue theorem implies 
\begin{align*}
   & 16\sum_{n\leq x}\lambda_{{\rm{sym}}^j  f}^2(n) l(n)\\
   &=\mathcal{C}_{f, j}' x^3+\frac{16}{2\pi i}\left\{\int_{3+\varepsilon-iT}^{3-\frac{1}{j^3}-iT}+\int_{3-\frac{1}{j^3}-iT}^{3-\frac{1}{j^3}+iT}+\int_{3-\frac{1}{j^3}+iT}^{3+\varepsilon+iT}\right\}\mathcal{F}_{2,j}(s)\frac{x^s}{s}ds\\
   &\qquad\qquad+O\left(\frac{x^{3+\varepsilon}}{T}\right)\\
    &:=\mathcal{C}_{f, j}' x^3+J_1+J_2+J_3+O\left(\frac{x^{3+\varepsilon}}{T}\right),
\end{align*}
where $\displaystyle\mathcal{C}_{f, j}' x^3=16\mathop{\mathrm{Res}}_{s=3}\mathcal{F}_{2,j}(s)\frac{x^s}{s}$.

Here we make the special choice $T=T^*$ of Lemma \ref{L2.9}, which satisfies \eqref{E14}, so that the horizontal portions $J_2$ and $J_3$ are controlled by the vertical line contribution $J_1$. The contributions of $J_1, J_2$ and $J_3$ are given by 
\begin{align*}
    J_1&\ll \int_{3-\frac{1}{j^3}-iT}^{3-\frac{1}{j^3}+iT} \left|\zeta(s-2)L(s, \chi)\prod_{n=1}^j L(s-2, {\rm{sym}^{2n}}f)L(s, {\rm{sym}^{2n}}f\otimes \chi)\right|\frac{x^s}{s}ds\\
    &\ll x^{3-\frac{1}{j^3}+\varepsilon}+x^{3-\frac{1}{j^3}+\varepsilon} \int_{10}^T\left|\zeta(1-\frac{1}{j^3}+it)\prod_{n=1}^j L(1-\frac{1}{j^3}+it, {\rm{sym}^{2n}}f)\right| t^{-1}dt\\
    &\ll x^{3-\frac{1}{j^3}+\varepsilon}T^{\frac{8\sqrt{15}}{63}\times\frac{1}{j^\frac{9}{2}}+\{\frac{(j+1)^2}{2}-\frac{4}{5} \}\frac{1}{j^3}-1+\varepsilon}
\end{align*}
and 
\begin{align*}
    &|J_2|+|J_3|\\
    &\ll \int_{3-\frac{1}{j^3}+iT}^{3+\varepsilon+iT}\left|\zeta(s-2)L(s, \chi)\prod_{n=1}^j L(s-2, {\rm{sym}^{2n}}f)L(s, {\rm{sym}^{2n}}f\otimes \chi)\right|\frac{x^s}{s}ds\\
    &\ll \int_{1-\frac{1}{j^3}}^{1+\varepsilon} \left|\zeta(\sigma+iT)\prod_{n=1}^j L(\sigma+iT, {\rm{sym}^{2n}}f)\right|x^{2+\sigma}T^{-1}d\sigma\\
    &\ll \frac{x^{3+\varepsilon}}{T}+x^{3-\frac{1}{j^3}}T^{\frac{(j+1)^2}{2}\frac{1}{j^3}-\frac{4}{5}\frac{1}{j^3}-1+\varepsilon},
\end{align*}
which follows from the proof of Theorem \ref{T1.1}.

Therefore, in total, we have
\begin{align*}
    16\sum_{n\leq x}\lambda_{{\rm{sym}}^j  f}^2(n) l(n)=&\mathcal{C}_{f, j}' x^3+O\left(x^{3-\frac{1}{j^3}+\varepsilon}T^{\frac{8\sqrt{15}}{63}\times\frac{1}{j^\frac{9}{2}}+\{\frac{(j+1)^2}{2}-\frac{4}{5} \}\frac{1}{j^3}-1+\varepsilon}\right)\\
    &+O\left(\frac{x^{3+\varepsilon}}{T}\right).
\end{align*}

Finally, making our choice of $T$ as $x^{3-\frac{1}{j^3}}T^{\frac{8\sqrt{15}}{63}\times\frac{1}{j^\frac{9}{2}}+\{\frac{(j+1)^2}{2}-\frac{4}{5} \}\frac{1}{j^3}-1}=\frac{x^{3}}{T}$, i.e.,  $T=x^{\frac{630j^\frac{3}{2}}{315j^\frac{3}{2}(j+1)^2-504j^\frac{3}{2}+80\sqrt{15}}}$, we obtain
\begin{equation}
 16\sum_{n\leq x}\lambda_{{\rm{sym}}^j  f}^2(n) l(n)=\mathcal{C}_{f, j}' x^3+O\left(x^{3-\frac{630j^\frac{3}{2}}{315j^\frac{3}{2}(j+1)^2-504j^\frac{3}{2}+80\sqrt{15}}+\varepsilon}\right).\label{E22}
\end{equation}
Now, for the sum $\displaystyle 4\sum_{n\leq x}\lambda_{{\rm{sym}}^j  f}^2(n) v(n)$, we apply Perron's formula to $\mathcal{F}_{3,j}(s)$, then, following Lemma \ref{L2.5}, we have 
\begin{align*}
   4  \sum_{n\leq x}\lambda_{{\rm{sym}}^j  f}^2(n) v(n)= \frac{4}{2\pi i}\int_{3+\varepsilon-iT}^{3+\varepsilon+iT}\mathcal{F}_{3,j}(s)\frac{x^s}{s}ds+O\left(\frac{x^{3+\varepsilon}}{T}\right), 
\end{align*}
where $10\leq T\leq x$ is a parameter to be chosen later, and $\mathcal{F}_{3,j}(s)$ is as in Lemma \ref{L2.5}.

We move the line of integration to $\Re(s)=3-\frac{1}{j^3}$, then, in the rectangle formed by the line segments joining the points $3+\varepsilon-iT,\ 3+\varepsilon+iT,\ 3-\frac{1}{j^3}+iT$ and $3-\frac{1}{j^3}-iT$, $\mathcal{F}_{3,j}(s)$ is a holomorphic function, and thus Cauchy's theorem implies
\begin{align*}
   & 4\sum_{n\leq x}\lambda_{{\rm{sym}}^j  f}^2(n) v(n)\\
   &=\frac{4}{2\pi i}\left\{\int_{3+\varepsilon-iT}^{3-\frac{1}{j^3}-iT}+\int_{3-\frac{1}{j^3}-iT}^{3-\frac{1}{j^3}+iT}+\int_{3-\frac{1}{j^3}+iT}^{3+\varepsilon+iT}\right\}\mathcal{F}_{3,j}(s)\frac{x^s}{s}ds+O\left(\frac{x^{3+\varepsilon}}{T}\right)\\
   &:=\widetilde{J}_1+\widetilde{J}_2+\widetilde{J}_3+O\left(\frac{x^{3+\varepsilon}}{T}\right).
\end{align*}
Here we make the special choice $T=T^*$ of Lemma \ref{L2.9}, so that we can use the argument that by meromorphic continuation, we get $L(\sigma+it,\chi)\ll |\zeta(\sigma+it)|\ll T^\varepsilon$ on the line $t=T^*$ for $\frac{1}{2}\leq \sigma\leq 2$.

Following the same arguments used for the estimation of $J_1$ above, we obtain the contribution of $\widetilde{J}_1$ as
\begin{align*}
    \widetilde{J}_1&\ll \int_{3-\frac{1}{j^3}-iT}^{3-\frac{1}{j^3}+iT}\left|\zeta(s)L(s-2,\chi)\prod_{n=1}^jL(s, {\rm{sym}}^{2n}  f)L(s-2, {\rm{sym}}^{2n}  f\otimes\chi)\right|\frac{x^s}{s}ds\\
    &\ll x^{3-\frac{1}{j^3}+\varepsilon}+x^{3-\frac{1}{j^3}+\varepsilon} \\
    &\qquad\int_{10}^T\left|L(1-\frac{1}{j^3}+it,\chi)\prod_{n=1}^j L(1-\frac{1}{j^3}+it, {\rm{sym}^{2n}}f)\right| t^{-1}dt\\
    &\ll x^{3-\frac{1}{j^3}+\varepsilon}T^{\frac{8\sqrt{15}}{63}\times\frac{1}{j^\frac{9}{2}}+\{\frac{(j+1)^2}{2}-\frac{4}{5} \}\frac{1}{j^3}-1+\varepsilon}.
\end{align*}

Now, the contributions of $\widetilde{J}_2$ and $\widetilde{J}_3$ are given by
\begin{align*}
    &|\widetilde{J}_2|+|\widetilde{J}_3|\\
    &\ll \int_{3-\frac{1}{j^3}+iT}^{3+\varepsilon+iT}\left|\zeta(s)L(s-2, \chi)\prod_{n=1}^j L(s, {\rm{sym}^{2n}}f)L(s-2, {\rm{sym}^{2n}}f\otimes \chi)\right|\frac{x^s}{s}ds\\
    &\ll \int_{1-\frac{1}{j^3}}^{1+\varepsilon} \left|L(\sigma+iT,\chi)\prod_{n=1}^j L(\sigma+iT, {\rm{sym}^{2n}}f\otimes \chi)\right|x^{2+\sigma}T^{-1}d\sigma\\
    &\ll \int_{1-\frac{1}{j^3}}^{1+\varepsilon} x^{2+\sigma} T^{\varepsilon+\frac{6}{5}(1-\sigma)+\frac{(j+1)^2-4}{2}(1-\sigma)-1}d\sigma\\
    &\ll \frac{x^{3+\varepsilon}}{T}+x^{3-\frac{1}{j^3}}T^{\frac{(j+1)^2}{2}\frac{1}{j^3}-\frac{4}{5}\frac{1}{j^3}-1+\varepsilon}.
\end{align*}

Therefore, in total, we have
\begin{align*}
    4\sum_{n\leq x}\lambda_{{\rm{sym}}^j  f}^2(n) v(n)\ll \frac{x^{3+\varepsilon}}{T}+x^{3-\frac{1}{j^3}+\varepsilon}T^{\frac{8\sqrt{15}}{63}\times\frac{1}{j^\frac{9}{2}}+\{\frac{(j+1)^2}{2}-\frac{4}{5} \}\frac{1}{j^3}-1+\varepsilon}.
\end{align*}
Finally, making our choice of $T$ as $x^{3-\frac{1}{j^3}}T^{\frac{8\sqrt{15}}{63}\times\frac{1}{j^\frac{9}{2}}+\{\frac{(j+1)^2}{2}-\frac{4}{5} \}\frac{1}{j^3}-1}=\frac{x^{3}}{T}$, i.e., $T=x^{\frac{630j^\frac{3}{2}}{315j^\frac{3}{2}(j+1)^2-504j^\frac{3}{2}+80\sqrt{15}}}$, we obtain
\begin{equation}
    4\sum_{n\leq x}\lambda_{{\rm{sym}}^j  f}^2(n) l(n)\ll x^{3-\frac{630j^\frac{3}{2}}{315j^\frac{3}{2}(j+1)^2-504j^\frac{3}{2}+80\sqrt{15}}}.\label{E23}
\end{equation}
By combining \eqref{E10}, \eqref{E22} and \eqref{E23}, we get the proof of Theorem \ref{T1.2}.

\section{Proof of Theorem \ref{T1.3}}\label{Sec5}
Let $j\geq 127$. Using the Lemmas \ref{L2.12}, \ref{L2.13}, and \ref{L2.14}, we improve the contribution of the integral $I_1$ in the proof of Theorem \ref{T1.1}, which consequently improves the error term. From the proof of Theorem \ref{T1.1}, we have 
\begin{align*}
   & 8\sum_{n\leq x}\lambda_{{\rm{sym}}^j  f}^2(n) r(n)\\
   &=\mathcal{C}_{f, j} x^2+\frac{8}{2\pi i}\left\{\int_{2+\varepsilon-iT}^{2-\frac{1}{\sqrt{j}}-iT}+\int_{2-\frac{1}{\sqrt{j}}-iT}^{2-\frac{1}{\sqrt{j}}+iT}+\int_{2-\frac{1}{\sqrt{j}}+iT}^{2+\varepsilon+iT}\right\}\mathcal{F}_{1,j}(s)\frac{x^s}{s}ds\\
   &\qquad+O\left(\frac{x^{2+\varepsilon}}{T}\right)\\
    &:=\mathcal{C}_{f, j} x^2+I_1+I_2+I_3+O\left(\frac{x^{2+\varepsilon}}{T}\right),
\end{align*}
where $\displaystyle\mathcal{C}_{f, j} x^2=8\mathop{\mathrm{Res}}_{s=2}\mathcal{F}_{1,j}(s)\frac{x^s}{s}$.

For $j\geq 127$, we have
\begin{align*}
    I_1&\ll \int_{2-\frac{1}{\sqrt{j}}-iT}^{2-\frac{1}{\sqrt{j}}+iT} \left|\zeta(s)\zeta(s-1)\prod_{n=1}^j L(s, {\rm{sym}^{2n}}f)L(s-1, {\rm{sym}^{2n}}f)\right|\frac{x^s}{s}ds\\
    &\ll x^{2-\frac{1}{\sqrt{j}}+\varepsilon}+x^{2-\frac{1}{\sqrt{j}}+\varepsilon} \int_{10}^T\left|\zeta(1-\frac{1}{\sqrt{j}}+it)\prod_{n=1}^j L(1-\frac{1}{\sqrt{j}}+it, {\rm{sym}^{2n}}f)\right| t^{-1}dt\\
    &\ll x^{2-\frac{1}{\sqrt{j}}+\varepsilon}+x^{2-\frac{1}{\sqrt{j}}+\varepsilon} \sup_{10\leq T_1\leq T}T_1^{-1}\left(\int_{10}^T|L(\sigma+it,{\rm{sym}}^{2j}f)|^\frac{12772}{1135}dt\right)^{\frac{1135}{12772}}\\
    &\left(\int_{T_1}^{2T_1}\mid \prod_{n=2}^{j-1} L(1-\frac{1}{\sqrt{j}}+it, {\rm{sym}^{2n}}f)\mid^{2}dt\right)^\frac{1}{2}  \\
    & \left(\int_{T_1}^{2T_1}|L(1-\frac{1}{\sqrt{j}}+it, {\rm{sym}^2f})|^\frac{12772}{5251}dt\right)^\frac{5251}{12772}\left\{\max_{T_1\leq t\leq 2T_1}\mid \zeta(1-\frac{1}{\sqrt{j}}+it)\mid\right\}\\
    &\ll x^{2-\frac{1}{\sqrt{j}}+\varepsilon} T^{\frac{8\sqrt{15}}{63}\frac{1}{j^\frac{3}{4}}+\frac{5251}{12772}+\frac{1135}{12772}+\frac{2j-1}{2}\frac{1}{\sqrt{j}}+\frac{j^2-4}{2}\frac{1}{\sqrt{j}}+\varepsilon}\\
    &\ll x^{2-\frac{1}{\sqrt{j}}+\varepsilon} T^{\frac{8\sqrt{15}}{63}\frac{1}{j^\frac{3}{4}}+\frac{1}{2}+\frac{(j+1)^2-6}{2}\frac{1}{\sqrt{j}}+\varepsilon},
\end{align*}
which follows from the Lemmas \ref{L2.7}, \ref{L2.8}, \ref{L2.10}, \ref{L2.13} and \ref{L2.14}. Note that here to use the Lemma \ref{L2.13}, suitably, we should have $1-\frac{1}{\sqrt{j}}>\frac{11637}{12772}$ and this holds only for $j\geq 127$. 

Following the similar arguments that are of Theorem \ref{T1.1}, we have 
\begin{align*}
   |I_2|+|I_3| &\ll \frac{x^{2+\varepsilon}}{T}+x^{2-\frac{1}{\sqrt{j}}}T^{\frac{(j+1)^2}{2}\frac{1}{\sqrt{j}}-\frac{4}{5}\frac{1}{\sqrt{j}}-1+\varepsilon}.
\end{align*}
Therefore, in total, we have
\begin{align*}
    8\sum_{n\leq x}\lambda_{{\rm{sym}}^j  f}^2(n) r(n)=&\mathcal{C}_{f, j} x^2+O\left( x^{2-\frac{1}{\sqrt{j}}+\varepsilon} T^{\frac{8\sqrt{15}}{63}\frac{1}{j^\frac{3}{4}}+\frac{1}{2}+\frac{(j+1)^2-6}{2}\frac{1}{\sqrt{j}}+\varepsilon}\right)\\
    &+O\left(\frac{x^{2+\varepsilon}}{T}\right).
\end{align*}
Finally, making our choice of $T$ as $x^{2-\frac{1}{\sqrt{j}}} T^{\frac{8\sqrt{15}}{63}\frac{1}{j^\frac{3}{4}}+\frac{1}{2}+\frac{(j+1)^2-6}{2}\frac{1}{\sqrt{j}}}=\frac{x^{2}}{T}$, i.e., $T=x^{\frac{126j^\frac{1}{4}}{63j^\frac{1}{4}(j+1)^2+63j^\frac{3}{4}-378j^\frac{1}{4}+16\sqrt{15}}}$, we obtain
\begin{align*}
    8\sum_{n\leq x}\lambda_{{\rm{sym}}^j  f}^2(n) r(n)=&\mathcal{C}_{f, j} x^2+O\left( x^{2-\frac{126j^\frac{1}{4}}{63j^\frac{1}{4}(j+1)^2+63j^\frac{3}{4}-378j^\frac{1}{4}+16\sqrt{15}}+\varepsilon}\right)
\end{align*}
for $j\geq 127$. 
This completes the proof of Theorem \ref{T1.3}.

\section{Proof of Theorem \ref{T1.4}}\label{Sec6}
The proof follows a similar approach to that of Theorem \ref{T1.3}, drawing on the arguments from the proof of Theorem \ref{T1.2}.

{\bf Acknowledgments}: The author acknowledges the financial support and research facilities provided by the Indian Institute of Science Education and Research (IISER) Berhampur, Department of Mathematical Sciences, during the tenure of the postdoctoral fellowship. The author is grateful to Prof. A. Sankaranarayanan for bringing Lemma \ref{L2.12} to their attention.


\begin{thebibliography}{}
\bibitem{Deligne} P. Deligne, \emph{La conjecture de Weil}, I, II, Publ. Math. IHES, \textbf{43} (1974) 273-308; IBID. \textbf{52} (1981) 313-428. 

\bibitem{Feng} J. Feng, \emph{On the second moment of the Fourier coefficients of symmetric power $L$-function over two sparse sequences of positive integers}. Period Math Hung (2025). 

\bibitem{Fomenko1999} O. M. Fomenko, \emph{Fourier coefficients of parabolic forms, and automorphic $L$-functions}. J. Math. Sci. (New York) \textbf{95} (1999), 2295–2316. 

\bibitem{Fomenko2006} O. M. Fomenko, \emph{Identities involving the coefficients of automorphic $L$-functions}. J. Math. Sci. \textbf{133}  (2006), 1749–1755.

\bibitem{Fomenko2008}  O. M. Fomenko,  \emph{Mean value theorems for automorphic $L$-functions}. St. Petersburg Math. J. \textbf{19} (2008), 853–866.

\bibitem{H+I} J. L. Hafner and A. Ivi\'c, \emph{On sums of Fourier coefficients of cusp forms.} Enseign. Math, {\bf 35(2)}, 375-382.


\bibitem{Hardy and Wright}  G. H. Hardy,  E. M. Wright,\emph{n Introduction to the Theory of Numbers}. Oxford University Press, Oxford (1979).

\bibitem{He} X. G. He, \emph{Integral power sums of Fourier coefficients of symmetric square $L$-functions}. Proc. Am. Math. Soc. \textbf{147} (2019),
2847–2856.

\bibitem{Heath-Brown} D. R. Heath-Brown.\emph{A new $k$th derivative estimate for a trigonometric sum via Vinogradov’s integral}. Tr. Mat. Inst. Steklova, 296:95–110, 2017. English version published in Proc. Steklov Inst. Math. \textbf{296}(1).
(2017), 88–103


\bibitem{HeckE1827} E. Hecke, \emph{Theorie der Eisensteinschen Reihen höherer Stufe und ihre Anwendung auf Funktionentheorie und Arithmetik}.
Abhandlungen Hamburg \textbf{5} (1927) 199–224.

\bibitem{Hua} G. D. Hua, \emph{The average behaviour of Hecke eigenvalues over certain sparse sequence of positive integers}. Res. Number Theory \textbf{5} (2022), no. 4, Paper No. 95, 20 pp.

\bibitem{Huang} B. R. Huang, \emph{On the Rankin–Selberg problem}. Math. Ann. \textbf{381}(3/4) (2021), 1217–1251.

\bibitem{Jiang and Lu}  Y. J. Jiang and  G. S. Lü, \emph{Uniform estimates for sums of coefficients of symmetric square $L$-function}. J. Number Theory \textbf{148} (2015), 220–234.

\bibitem{Kratzel} E. Krätzel, \emph{Analytische Funktionen in der Zahlentheorie}. Teubner-Texte zur Mathematik 139. B.G.Teubner, Stuttgart, 2000.

\bibitem{LauLu} Y.-K. Lau and G. S. Lü,  \emph{Sums of Fourier coefficients of cusp forms}. Q. J. Math. \textbf{62}, (2011) 687–716.

\bibitem{LauLuWu} Y.-K. Lau, G. S. Lü, J. Wu, \emph{Integral power sums of Hecke eigenvalues}. Acta Math. 150 (2011), 193–207.



\bibitem{LinNunesQi} Y. X. Lin, R. Nunes, and Z. Qi, \emph{Strong subconvexity for self-dual $GL(3)$ $L$-functions}, Int. Math. Res. Not., {\bf 13} (2023) 11453-11470.

\bibitem{Liu}  H. F. Liu, \emph{On the asymptotic distribution of Fourier coefficients of cusp forms}. Bull. Braz. Math. Soc. (N.S.) \textbf{54} (2023), 17.

\bibitem{Liu and Yang} H. Liu and X. Yang, \emph{The average behaviors of the Fourier coefficients of $j$th symmetric power $L$-function over two sparse sequences of positive integers}, Bull. Iranian Math. Soc. \textbf{50} (2024), no. 1, Paper No. 14, 16 pp.

\bibitem{Lu2009(1)}  G.S. Lü,  \emph{Average behavior of Fourier coefficients of cusp forms}. Proc. Am. Math. Soc. \textbf{137}(6) (2009), 1961–1969.

 \bibitem{Lu2009(2)} G.S. Lü,  \emph{The sixth and eighth moments of Fourier coefficients of cusp forms}. J. Number Theory \textbf{129} (2009), 2790–2800.
 
\bibitem{Lu2011}G.S. Lü,  \emph{On higher moments of Fourier coefficients of holomorphic cusp forms}. Can. J. Math. \textbf{63} (2011), 634–647.



\bibitem{Luo Lao and Zou} S. Luo, H. X. Lao, and A. Y. Zou, \emph{Asymptotics for Dirichlet coefficients of symmetric power $L$-functions}. Acta Math. \textbf{199} (2021), 253–268.

 \bibitem{Newton and Thorne2021} J. Newton and J. A. Thorne, \emph{Symmetric power functoriality for holomorphic modular forms}, Publ. Math. Inst. Hautes Études Sci.   \textbf{134} (2021), 1–116. 

\bibitem{Newton and Thorne2021 II}  J. Newton and J. A. Thorne, \emph{Symmetric power functoriality for holomorphic modular forms II}, Publ. Math. Inst. Hautes Études Sci.   \textbf{134} (2021), 117–152. 

\bibitem{Perelli} A. Perelli, \emph{General $L$-functions}, Ann. Mat. Pura Appl., {\bf 130} (1982) 287-306.

\bibitem{KRAS} K. Ramachandra and A. Sankaranarayanan, \emph{Notes on the Riemann zeta-function}, Journal of Indian Math. soc., \textbf{57} (1991) 67-77.

\bibitem{Rankin} R. A. Rankin, \emph{Contributions to the theory of Ramanujan’s function $\tau(n)$ and similar arithmetical functions II. The order of the Fourier coefficients of the integral modular forms}. Proc. Camb. Philos. Soc \textbf{35} (1939), 357–372.

\bibitem{AS Singh and Srinivas} A. Sankaranarayanan,  S. K. Singh and K. Srinivas, \emph{ Discrete mean square estimates for coefficients of symmetric power $L$-functions}. Acta Arith. \textbf{190} (2019), 193–208.

\bibitem{ASAS2022} A. Sharma and A. Sankaranarayanan, \emph{Discrete mean square of the coefficients of symmetric square $L$-functions on certain sequence of positive numbers}, Res. Number Theory, {\bf 8} (2022) 19.

\bibitem{ASAS2023} A. Sharma and A. Sankaranarayanan, \emph{On the average behavior of the Fourier coefficients of $j$th symmetric power $L$-function over certain sequences of positive integers}, Czech, Math. J., {\bf 73} (2023) 885-901.

\bibitem{Selberg} A. Selberg, \emph{Bemerkungen über eine Dirichletsche, die mit der Theorie der Modulformen nahe verbunden ist}. Arch.
Math. Naturvid. \textbf{43} (1940), 47–50 .

\bibitem{Tang and Wu}  H. C. Tang and J. Wu, \emph{Fourier coefficients of symmetric power $L$-functions}. J. Number Theory \textbf{167} (2016), 147–160.

\bibitem{Titchmarsh} E. C. Titchmarsh and D. R. Heath-Brown, \emph{The Theory of the Riemann Zeta-function}, Second edition. The Clarendon Press, Oxford University Press, New York, 1986. 

\bibitem{Walfisz} A. Walfisz, \emph{über die Koeffizientensummen einiger Modulformen.} Mathematische Annalen {\bf 108.1} (1933): 75-90.


\bibitem{Wang} Y. Wang, \emph{A note on average behaviour of the Fourier coefficients of $j^\textit{th}$ symmetric power $L$-function over certain sparse sequence of positive integers}, Czech. Math. J., {\bf 74} (2024) 623–636.

\bibitem{Wu} J. Wu, \emph{Power sums of Hecke eigenvalues and application}, Acta Arith. {\bf 137} (2009), 333-344.

\bibitem{Xu} C. R. Xu, \emph{General asymptotic formula of Fourier coefficients of cusp forms over sum of two squares}. J. Number Theory \textbf{236} (2022), 214–229.

\bibitem{Zhai} S. Zhai, \emph{Average behavior of Fourier coefficients of cusp forms over sum of two squares}, J. Number Theory \textbf{133} (2013), no. 11, 3862–3876.
























\end{thebibliography}
\end{document}